\documentclass{amsart}

\usepackage{latexsym}   
\usepackage{amssymb, amsmath, amsthm}    
\usepackage{subcaption}
\usepackage{graphicx}
\usepackage{enumerate}
\usepackage{tikz}
\usetikzlibrary{automata,positioning,shapes,snakes}
\usetikzlibrary{decorations.markings}
\tikzstyle arrowstyle=[scale=1]
\tikzstyle directed=[postaction={decorate,decoration={markings,mark=at position .65 with {\arrow[arrowstyle]{stealth}}}}]

\newtheorem{theorem}{Theorem}[section]
\newtheorem{lemma}[theorem]{Lemma}

\newtheorem{remark}[theorem]{Remark}
\newtheorem{example}[theorem]{Example}
\newtheorem{definition}[theorem]{Definition}
\input diagxy

\def\ZZ{{\mathbb{Z}}}

\def\CP{{\mathcal P}}
\def\CL{{\mathcal L}}

\def\CT{{\mathcal T}}
\def\CF{{\mathcal F}}

\newtheorem{question}[theorem]{Open Question}

\begin{document}

\begin{center}
\uppercase{\bf Cyclic Critical Groups of Graphs}

\vskip .25in

{\bf Ryan Becker}\\
{\small Department of Mathematics, Colorado State University, Fort Collins, CO}\\
{\tt becker@math.colostate.edu}\\

\vskip .1in

{\bf Darren B Glass}\\
{\small Department of Mathematics, Gettysburg College, Gettysburg PA 17325}\\
{\tt dglass@gettysburg.edu}\\

\end{center}

\begin{abstract}
In this note, we describe a construction that leads to families of graphs whose critical groups are cyclic.  For some of these families we are able to give a formula for the number of spanning trees of the graph, which then determines the group exactly.  We also pose several open questions related to this work.
\end{abstract}

\section{Introduction and Background}

This article will discuss some results related to a solitaire chip-firing game played on the vertices of a finite connected graph $G$. In order to describe this game, let us first define a {\em configuration} on the graph $G$ to be an assignment of an integer number of chips to each vertex of $G$. These numbers can be positive, negative, or zero, and if we denote a configuration by $\delta$ then $\delta(v)$ will be the number of chips assigned to the vertex $v$.  Given a configuration, we define its {\em degree} to be the total number of chips assigned.

We next define the legal transitions between configurations, by letting a {\em move} consist of choosing a vertex and either borrowing one chip from each adjacent vertex or firing one chip to each adjacent vertex. See Figure \ref{F:move} for one example. We note that borrowing at a vertex is equivalent to firing at all other vertices simultaneously and vice versa; we allow both as a matter of convenience. We will say that two configurations are {\em equivalent} if one can get from one to the other through a sequence of these moves. It is clear that a necessary but not sufficient condition for two configurations to be equivalent is that they have the same degree.

\begin{figure}[h]
\centering
\begin{subfigure}[b]{.4\textwidth}
\begin{tikzpicture}
[scale = 2, every node/.style={circle,minimum size=0.8cm,fill=blue!20}]
\node (v1) at (0,0) {0};
\node (v2) at (1,0) {4};
\node (v3) at (2,-0.5) {-1};
\node (v4) at (2,0.5) {-1};
\draw[directed] (v2) -- (v1);
\draw[directed] (v2) -- (v3);
\draw[directed] (v2) -- (v4);
\draw (v3) -- (v4);
\end{tikzpicture}
\caption{Before firing the center vertex}
\end{subfigure} \qquad
\begin{subfigure}[b]{.4\textwidth}
\begin{tikzpicture}
[scale = 2, every node/.style={circle,minimum size=0.8cm,fill=blue!20}]
\node (v1) at (0,0) {1};
\node (v2) at (1,0) {1};
\node (v3) at (2,-0.5) {0};
\node (v4) at (2,0.5) {0};
\draw (v1) -- (v2);
\draw (v2) -- (v3);
\draw (v2) -- (v4);
\draw (v3) -- (v4);
\end{tikzpicture}
\caption{After firing the center vertex}
\end{subfigure}
\caption{Configurations on a graph before and after firing the center vertex}
\label{F:move}
\end{figure}
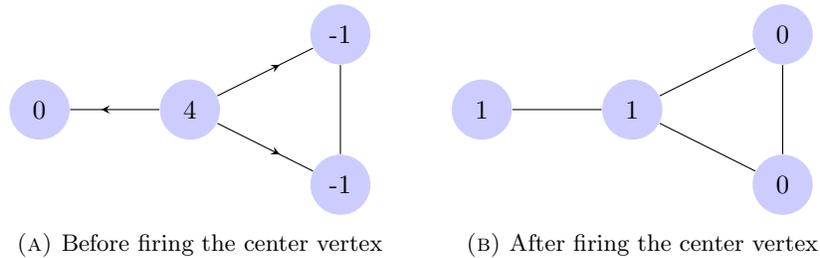

This setup may appear to be purely combinatorial in nature but it has a number of interesting applications in areas such as statistical physics, cryptography, algebraic geometry, and economics.  We define the {\em critical group} of $G$ to be the set of equivalence classes of configurations with degree zero. This set is naturally endowed with an abelian group structure where the group operation is addition of chips at corresponding vertices.  We will denote this group by $K(G)$.  Due to analogies with the set of divisors on an algebraic curve up to linear equivalence, this group is also known as the Jacobian of the graph $G$.  For more details on these connections to algebraic geometry, we refer the reader to \cite{BN1}.

It is well-known that for a given graph on $n$ vertices, one can compute its critical group by noting that the set of configurations of degree zero is isomorphic to $\ZZ^{n-1}$.  The critical group of $G$ is then isomorphic to $\ZZ^{n-1}/Im(\CL^*)$, where $\CL^*$ is the reduced Laplacian matrix of the graph $G$ (see \cite{Biggs},  \cite{Merino}, \cite{Primer} for details).  As discussed in \cite{Lor1}, one can compute the group structure of this quotient by computing the Smith Normal Form of the matrix $\CL^*$.  While efficient algorithms to do this are known (see, for example, \cite{EGV}) they often do not take into account the combinatorial structure of the graph.  Several recent papers including \cite{BMMPR}, \cite{CR}, \cite{GM}, and \cite{RT} attempt to use this structure in order to gain some insight into critical groups.  Some of these results use the fact that the order of the critical group of a graph is equal to the number of spanning trees of that graph, which is a corollary of Kirchhoff's Matrix Tree Theorem.  One result that is well known (see, for example, \cite[Prop 1.2]{CR}) and which we will use repeatedly is the following:

\begin{lemma}\label{L:sum}
Let $G_1$ and $G_2$ be two graphs and let $H$ be the graph obtained by identifying a single vertex of $G_1$ with a single vertex of $G_2$.  Then the critical group of $H$ is isomorphic to the direct sum of the critical groups of $G_1$ and $G_2$.
\end{lemma}

Given a graph $G$, it is natural to ask what the minimal number of elements needed to generate the critical group of $G$ is.  The extreme cases are handled by letting $G$ be a tree, in which case the critical group is trivial, and letting $G$ be the complete graph $K_n$, in which case the critical group is $(\ZZ/n\ZZ)^{n-2}$.  We also note that for any finite abelian group $\Gamma \cong \ZZ/m_1\ZZ \oplus \ldots \oplus \ZZ/m_r\ZZ$ it is possible to construct a graph $G$ whose critical group is $\Gamma$ by starting with $k$ cycles of length $m_1,\ldots , m_k$ and identifying a single vertex on each of the cycles by Lemma \ref{L:sum}.  While this construction shows that the rank of the critical group of a graph can be arbitrarily large, Wagner conjectured in \cite[Conj 4.2]{Wag} that the probability that a suitably defined random graph has a cyclic critical group approaches one. While this conjecture has recently been shown to be false, and Wood shows in \cite[Cor 9.5]{Wood} that the probability that a random graph has cyclic critical group is less than $0.8$, there is still significant evidence that most random graphs have cyclic critical groups.  In this note we will construct large families of graphs for which the critical group will be cyclic and in some cases we will be able to compute the order of this cyclic group.

\section{Adding Chains To Graphs}

Given a graph $G$ and two vertices $x,y \in V(G)$ we define $\delta_{x,y}$ to be the configuration on $G$ so that $\delta_{x,y}(x)=1, \delta_{x,y}(y)=-1$ and $\delta_{x,y}(v)=0$ for $v \ne x,y$.  We note that $\delta_{x,y} = -\delta_{y,x}$, and in particular the two divisors will generate the same subgroup of $K(G)$.

\begin{definition}
A {\it generating pair} of vertices for a graph $G$ is a pair $\{x,y\} \subset V(G)$ so that the configuration $\delta_{x,y}$ is a generator of the critical group of $G$.  Equivalently, $\{x,y\}$ will be a generating pair if any configuration of degree zero is equivalent to a configuration which has value zero except possibly at $x$ and $y$.
\end{definition}

\begin{example}
Let $G$ be an $n$-cycle.  More explicitly, let $G$ be a graph with $V(G)=\{x_1,\ldots,x_n\}$ and an edge between $x_i$ and $x_j$ if and only if $i \equiv j \pm 1$ mod $n$.  Let $\delta$ be any configuration of total degree $0$ on $G$.  We claim that $\delta$ is equivalent to a multiple of $\delta_{x_{n-1},x_n}$.

To see this, we let $\delta_1$ be the configuration obtained from $\delta$ by borrowing $\delta(x_1)$ times at the vertex $x_2$.  In particular, $\delta_1$ will be the configuration defined by setting $\delta_1(x_1)=0, \delta_1(x_2)=\delta(x_2)+2\delta(x_1), \delta_1(x_3)=\delta(x_3)-\delta(x_1)$, and $\delta_1(x_i)=\delta(x_i)$ for all $i \ge 4$.  For each $2 \le k \le n-2$ we define $\delta_k$ inductively as the configuration obtained from $\delta_{k-1}$ by borrowing $\delta_{k-1}(x_k)$ times at $x_{k+1}$.

We note that the configuration $\delta_{n-2}$ is equivalent to $\delta$ and $\delta_{n-2}(x_i)=0$ except possibly at $i=n-1,n$.  This verifies our claim and in particular proves that $\{x_{n-1},x_n\}$ is a generating pair for $G$.  More generally, one can show that the pair $\{x_i,x_j\}$ is a generating pair if and only if $gcd(i-j,n)=1$.
\end{example}

It is not always the case that a generating pair consists of two adjacent vertices.  For example, if $G$ is the graph in Figure \ref{F:tripent} it follows from Lemma \ref{L:sum} $K(G) \cong \ZZ/15\ZZ$ but that $\delta_{x,y}$ will either have order three or five for any pair of adjacent vertices.  However, for the vertices labelled $a$ and $b$ one can see that $\delta_{a,b}$ will generate the full group.

\begin{figure}[h]
\centering
\begin{subfigure}[b]{.35\textwidth}
\centering
\begin{tikzpicture}
[scale = 1.2, every node/.style={circle,minimum size=0.5cm,fill=blue!20}]
\node (v1) at (.5,0) {};
\node (v2) at (1.5,0) {$b$};
\node (v3) at (2,1) {};
\node (v4) at (0,1) {};
\node (v5) at (1,2) {};
\node (v6) at (0,3) {$a$};
\node (v7) at (2,3) {};
\draw (v1) -- (v2);
\draw (v3) -- (v2);
\draw (v4) -- (v1);
\draw (v3) -- (v5);
\draw (v4) -- (v5);
\draw (v6) -- (v7);
\draw (v5) -- (v6);
\draw (v5) -- (v7);
\end{tikzpicture}
\caption{A graph with cyclic critical group and no adjacent generating pairs}
\label{F:tripent}
\end{subfigure}
\qquad
\begin{subfigure}[b]{.35\textwidth}
\centering
\begin{tikzpicture}
[scale = 1.2, every node/.style={circle,minimum size=0.5cm,fill=blue!20}]
\node (v1) at (.5,0) {};
\node (v2) at (1.5,0) {};
\node (v3) at (2,1) {};
\node (v4) at (0,1) {};
\node (v5) at (1,2) {z};
\node (v6) at (.25,3) {};
\node (v7) at (-.5,2) {};
\node (w1) at (1,3){};
\node (w2) at  (1.5,3.5){};
\node (w3) at (2.2,3.8){};
\node (w4) at (2.8,3.2){};
\node (w5) at (2.5,2.5){};
\node (w6) at (2,2){};
\draw (v1) -- (v2);
\draw (v3) -- (v2);
\draw (v4) -- (v1);
\draw (v3) -- (v5);
\draw (v4) -- (v5);
\draw (v6) -- (v7);
\draw (v5) -- (v6);
\draw (v5) -- (v7);
\draw (w1)--(w2);
\draw (w2)--(w3);
\draw (w3)--(w4);
\draw (w4)--(w5);
\draw (w5)--(w6);
\draw (w6)--(v5);
\draw (w1)--(v5);
\end{tikzpicture}
\caption{A graph with cyclic critical group and no generating pairs}
\label{F:tripentsept}
\end{subfigure}
\caption{Examples}
\end{figure}
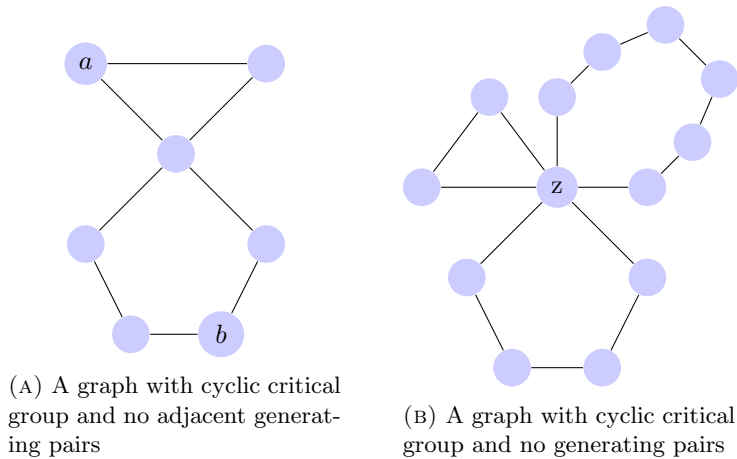

We note that even in a situation where a graph has a cyclic critical group then there does not need to be a generating pair.   The following example describes such a situation, answering a question posed by Lorenzini in \cite[Remark 2.11]{Lor3}.

\begin{example}
Let $G$ be the graph in Figure \ref{F:tripentsept}.  By Lemma \ref{L:sum},  $K(G) \cong \ZZ/105\ZZ$.  Moreover, if $z$ is the labelled vertex and $x \ne z$ is a different vertex on a cycle of size $d_x \in \{3,5,7\}$ then we note that the divisor $\delta_{x,z}$ has order $d_x$.  For any two vertices $x,y$ both of which are distinct from $z$, the divisor  $\delta_{x,y}$ can be written as $\delta_{x,z}-\delta_{y,z}$, and therefore has order equal to $lcm(d_x,d_y) \in \{3,5,7,15,21,35\}$ and in particular not equal to $|K(G)|$.
\end{example}

In the situation where our graph has a known generating pair, then we are able to construct a family of graphs which also have cyclic critical groups and known generating pairs due to the following theorem, which is the main result of this section.

\begin{theorem}\label{T:path}
Let $x$ and $y$ be a generating pair for $G$.  Let $\tilde{G}$ be the graph $G$ with an additional path of $\ell \ge 1$ edges (and $\ell-1$ new vertices) between the vertices $x$ and $y$.  Then any pair of consecutive vertices along this path are a generating pair for $\tilde{G}$. In particular, $K(\tilde{G})$ is cyclic.
\end{theorem}

\begin{proof}
Let $G$ be a graph and $\{x,y\}$ be a generating pair for $G$.  In particular, this means that for any configuration $\delta$ on $G$ we can do a series of moves so that the resulting configuration has chips only on $x$ and $y$.

Let $\tilde{G}$ be the graph with an additional path of length $\ell$ between vertices $x$ and $y$.  To be precise, $V(\tilde{G}) = V(G) \cup \{x_1,\ldots,x_{\ell-1}\}$ and the edges of $\tilde{G}$ will be the edges of $G$ along with edges connecting $x_i$ and $x_{i+1}$ for $1 \le i \le \ell-2$ as well as edges connecting $x$ to $x_1$ and $x_{\ell-1}$ to $y$. By convention, we set $x_0=x$ and $x_\ell=y$.

Given a configuration $\tilde{\delta}$ on $\tilde{G}$ we can consider its restriction $\tilde{\delta}|_G$ as a configuration (not necessarily of degree zero) on $G$.  We know there exists a sequence of legal moves that will make this configuration have chips only on the two vertices $x$ and $y$.  We perform this sequence of moves on $\tilde{\delta}$ and denote the resulting configuration on $\tilde{G}$ by $\delta_0$.

We have now moved all of the chips in the configuration onto the chain connecting $x$ and $y$, and we can therefore consolidate these on any two adjacent vertices.  To be explicit, choose two adjacent vertices $x_i$ and $x_{i+1}$.  If $i \ge 1$ then for each $1 \le j \le i$ we let $\delta_j$ be the configuration obtained by borrowing $\delta_{j-1}(x_{j-1})$ times at the vertex $x_j$.  In particular, the configuration $\delta_i$ will only have a nonzero value for vertices in $\{x_i,\ldots,x_\ell\}$.

We continue by defining $\delta_j$ for $j > i$.  In particular, for each $i<j \le  \ell-1$ we let $\delta_j$ be the configuration obtained by borrowing $\delta_{j-1}(x_{\ell-j})$ times at the vertex $x_{\ell-j-1}$. At the end of this process, the resulting configuration $\delta_{\ell-1}$ will only have a nonzero number of chips on the vertices $x_i$ and $x_{i+1}$.  In particular, we have shown that every configuration on $\tilde{G}$ of degree zero is equivalent to a multiple of the divisor $\delta_{x_i,x_{i+1}}$ and therefore $\{x_i,x_{i+1}\}$ is a generating pair for $\tilde{G}$.
\end{proof}

We note that Theorem \ref{T:path} is also a consequence of results in \cite[Sect.2]{Krep}.  However, our proof is more elementary.

\begin{example}\label{Ex:house}
Let $G$ be the `house' graph as pictured in Figure \ref{F:house} with vertices as labelled.  Assume that $\delta$ is a configuration of total degree zero on $G$. The fact that a $3$-cycle has cyclic critical group and that any pair of adjacent vertices is a generating pair for the graph tells us that there is a sequence of moves that will lead to an equivalent divisor $\delta_1$ with $\delta_1(z)=0$.  In particular, we can let $\delta_1$ be the divisor obtained by borrowing $\delta(z)$ times at the vertex $x$.

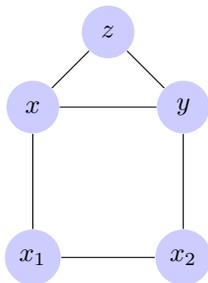
\begin{figure}[h]
\centering
\begin{tikzpicture}
[scale = 1, every node/.style={circle,minimum size=0.7cm,fill=blue!20}]
\node (v1) at (0,0) {$x_1$};
\node (v2) at (2,0) {$x_2$};
\node (v3) at (2,2) {$y$};
\node (v4) at (0,2) {$x$};
\node (v5) at (1,3) {$z$};
\draw (v1) -- (v2);
\draw (v3) -- (v2);
\draw (v3) -- (v4);
\draw (v4) -- (v1);
\draw (v3) -- (v5);
\draw (v4) -- (v5);
\end{tikzpicture}
\caption{The one-story house is one simple example of a stack of polygons.}
\label{F:house}
\end{figure}

If we now let $\gamma$ be the divisor obtained by borrowing $\delta_1(x)$ times at the vertex $x_1$ and $\delta_1(y)$ times at the vertex $x_2$, we can check that $\gamma(v)$ is only nonzero at $x_1,x_2$. In particular, $(x_1,x_2)$ is a generating pair for $G$.  In a similar manner, we could show that $(x,x_1)$ and $(x_2,y)$ are also generating pairs for $G$.
\end{example}

One can generalize the construction in Example \ref{Ex:house} to more general stacks of polygons.  In particular, let $(k_1,\ldots,k_n)$ be a sequence of integers with each $k_i \ge 2$.  Define the graph $G_1$ to be a $k_1$-cycle and, for each $1<i\le n$ define the graph $G_i$ by starting with graph $G_{i-1}$ and adding a path of $k_i-1$ edges between any two consecutive vertices of the path added at the previous step.  The resulting graph $G_n$ will consist of a stack of polygons with $k_1,\ldots,k_n$ sides.  One example is that the stack corresponding to $(3,4)$ or $(4,3)$ are isomorphic to the house graph in Example \ref{Ex:house}.  See Figure \ref{F:stacks} for additional examples.  It follows from inductive applications of Theorem \ref{T:path} that $K(G_n)$ is cyclic; we note that similar results are discussed in \cite{Lor2}.

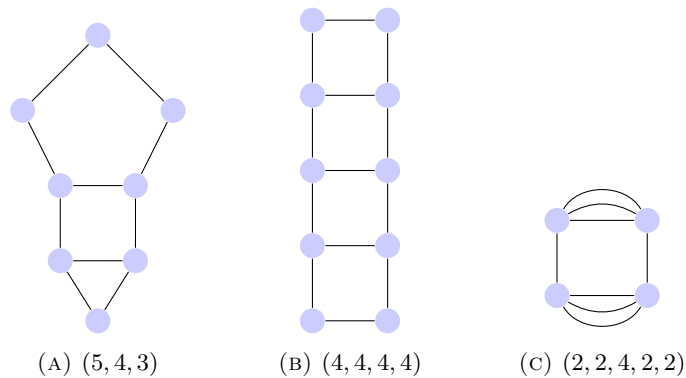
\begin{figure}[h]
\begin{subfigure}[b]{.2\textwidth}
\centering
\begin{tikzpicture}
[scale = 1, every node/.style={circle,minimum size=0.3cm,fill=blue!20}]
\node (v1) at (.5,0) {};
\node (v2) at (1.5,0) {};
\node (v3) at (2,1) {};
\node (v4) at (0,1) {};
\node (v5) at (1,2) {};
\node (v6) at (.5,-1) {};
\node (v7) at (1.5,-1) {};
\node (v8) at (1,-1.8) {};
\draw (v1) -- (v2);
\draw (v3) -- (v2);
\draw (v4) -- (v1);
\draw (v3) -- (v5);
\draw (v4) -- (v5);
\draw (v1) -- (v6);
\draw (v2) -- (v7);
\draw (v6) -- (v8);
\draw (v7) -- (v8);
\draw (v6) -- (v7);
\end{tikzpicture}
\caption{$(5,4,3)$}
\end{subfigure} \qquad
\begin{subfigure}[b]{.2\textwidth}
\centering
\begin{tikzpicture}
[scale = 1, every node/.style={circle,minimum size=0.3cm,fill=blue!20}]
\node (a1) at (0,0) {};
\node (a2) at (0,1) {};
\node (a3) at (0,2) {};
\node (a4) at (0,3) {};
\node (a5) at (0,4) {};
\node (b1) at (1,0) {};
\node (b2) at (1,1) {};
\node (b3) at (1,2) {};
\node (b4) at (1,3) {};
\node (b5) at (1,4) {};
\foreach \from/\to in {a1/a2,a2/a3,a3/a4,a4/a5,a1/b1,a2/b2,a3/b3,a4/b4,a5/b5,b1/b2,b2/b3,b3/b4,b4/b5}
    \draw (\from) -- (\to);
\end{tikzpicture}
\caption{$(4,4,4,4)$}
\end{subfigure}
\qquad
\begin{subfigure}[b]{.2\textwidth}
\centering
\begin{tikzpicture}
[scale = 1, every node/.style={circle,minimum size=0.3cm,fill=blue!20}]
\node (a1) at (0,2) {};
\node (a2) at (0,3) {};
\node (b1) at (1.2,2) {};
\node (b2) at (1.2,3) {};
\foreach \from/\to in {a1/a2,a1/b1,a2/b2,b1/b2}
    \draw (\from) -- (\to);
\draw (a2) to[bend left=30] (b2);
\draw (a2) to[bend left=60] (b2);
\draw (a1) to[bend right=30] (b1);
\draw (a1) to[bend right=60] (b1);
\end{tikzpicture}
\caption{$(2,2,4,2,2)$}
\end{subfigure}
\caption{Polygonal stacks corresponding to $(k_1,\ldots,k_n)$}
\label{F:stacks}
\end{figure}

We conclude this section by discussing some similarities between our result and results of Dino Lorenzini.  In particular, \cite[Thm 5.1]{Lor} gives the following result:

\begin{theorem}\label{T:Lor1}
Let $G$ be a connected graph with vertices $x,y$ so that there are $c>0$ edges between $x$ and $y$.  Moreover, let $G_1$ be the graph obtained by deleting all edges between the two vertices $x$ and $y$.  If $|K(G)|$ and $|K(G_1)|$ are relatively prime then $K(G)$ is cyclic.
\end{theorem}

In \cite{Lor2}, he gives an alternate proof of this theorem and strengthens the result somewhat.  In particular, he is able to prove:

\begin{theorem}\label{T:Lor2}
Let $G$ be a connected graph with vertices $x,y$ connected by at least one edge so that $|K(G)|$ and $|K(G_1)|$ are relatively prime, where $G_1$ is as defined in the previous theorem.  Let $G'$ be the graph obtained from $G$ by adding a path of $\ell$ edges between $x$ and $y$, and let $G_1'$ be the graph obtained from $G'$ by deleting the single edge between any two adjacent vertices in the chain.  Then $|K(G_1)|$ and $|K(G_1')|$ are relatively prime.  In particular, it follows from Theorem \ref{T:Lor1} that $K(G')$ is cyclic.
\end{theorem}

We note the similarities between Theorem \ref{T:Lor2} and Theorem \ref{T:path}.  This leads us to pose the following question.

\begin{question}
Given a graph $G$ and a pair of vertices $x,y$ so that $|K(G)|$ and $|K(G_1)|$ are relatively prime, must it be the case that the configuration $\delta_{x,y}$ is a generator of $K(G)$?
\end{question}

\section{Recurrence Relations and Orders of Critical Groups}\label{S:count}

Given a string of integers $k_1,\ldots,k_n$ with all $k_i>1$, we define $G_n$ to be a stack of polygons given by the construction in the previous section.  Such a graph is not uniquely defined by the $n$-tuple, as we could stack the polygons along different edges and get different graphs.  However, we will see in this section that all such graphs will have the same critical group.  In particular, it follows from Theorem \ref{T:path} that $K(G_n)$ is a cyclic group.  Moreover, it is a consequence of the Matrix Tree Theorem (see \cite{BHN}) that the order of the critical group of any graph is equal to the number of spanning trees of the graph.  Therefore if we can count the number of spanning trees of $G_n$ then we will know the critical group up to isomorphism.

In order to count spanning trees on our polygonal graphs, we use the notion of a spanning forest, introduced in \cite{DM}. A {\em spanning forest} is a pair of disjoint trees on a graph with specified roots that together span all vertices in the graph.  For example, if we fix two adjacent vertices of a $4$-cycle then there are three spanning forests of this graph, as seen in Figure \ref{F:Forests}.

\begin{figure}[h]
\centering
\begin{tikzpicture}
[scale = 2, every node/.style={circle,minimum size=0.5cm,fill=blue!20}]
\node (v1) at (0,0) {$x$};
\node (v2) at (1,0) {$y$};
\node (v3) at (1,1) {};
\node (v4) at (0,1) {};
\draw (v1) -- (v4);
\draw (v3) -- (v4);
\draw[dashed] (v1) -- (v2);
\draw[dashed] (v3) -- (v2);
\node (v5) at (2,0) {$x$};
\node (v6) at (3,0) {$y$};
\node (v7) at (3,1) {};
\node (v8) at (2,1) {};
\draw[dashed] (v5) -- (v6);
\draw (v6) -- (v7);
\draw (v7) -- (v8);
\draw[dashed] (v8) -- (v5);
\node (v9) at (4,0) {$x$};
\node (v10) at (5,0) {$y$};
\node (v11) at (5,1) {};
\node (v12) at (4,1) {};
\draw (v9) -- (v12);
\draw[dashed] (v9) -- (v10);
\draw (v10) -- (v11);
\draw[dashed] (v11) -- (v12);
\end{tikzpicture}
\caption{The $3$ spanning forests on $C_4$ rooted at $x$ and $y$.}
\label{F:Forests}
\end{figure}
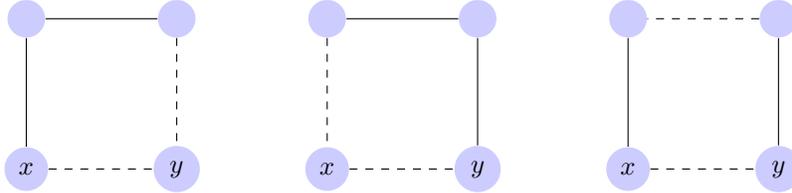

\begin{definition}
Let $G_n$ be the graph as described above, and for each $i=1,\dots,n-1$ let $e_i$ be the edge shared by the $k_i$-gon and the $k_{i+1}$-gon in the stack.
We define $T(k_1,\ldots,k_n)$ to be the number of spanning trees on the graph $G_n$ defined by the $n$-tuple $(k_1,\ldots,k_n)$ as above.  Additionally, it is easy to see that for any pair of consecutive vertices $x$ and $y$ on the $k_n$-gon ending the stack of polynomials other than the two vertices connected by $e_{n-1}$, the number of spanning forests on $G_n$ which are rooted at $x$ and $y$ will be the same as they must contain all edges of this polygon other than $e_{n-1}$ and the edge between $x$ and $y$; we denote the number of such forests by $F(k_1,\ldots,k_n)$.
\end{definition}

For example, if one considers the `house' graph from Figure \ref{F:house}, one can compute that there are eleven spanning trees so that $T(3,4) = 11$.  Moreover, there are eight spanning forests rooted at the vertices $x_1$ and $x_2$, so $F(3,4) = 8$.

\begin{lemma}\label{L:two}
The functions $F$ and $T$ are related by the following recurrence relations.

\[T(k_1,\ldots, k_n) = (k_n-1)T(k_1,\ldots,k_{n-1}) + F(k_1,\ldots,k_{n-1})\]
\[F(k_1,\ldots, k_n) = (k_n-2)T(k_1,\ldots,k_{n-1}) + F(k_1,\ldots,k_{n-1})\]
\end{lemma}

\begin{proof}

Let $G_{n-1}$ be a graph defined by the $(n-1)$-tuple $(k_1,\ldots,k_{n-1})$, and let $x$ and $y$ be two vertices on the $(n-1)^{st}$ level of the graph connected by an edge.  Let $\CP$ be a path of $k_n-1$ additional edges joining $x$ and $y$. We label the new vertices $x_2,\ldots,x_{k_n-1}$, letting $x_1=x$ and $x_{k_n}=y$ by convention.  Define the graph $G_n$ be the union $G_{n-1} \cup \CP$.

Let $\CT$ be a spanning tree on $G_n$.  If $\CT$ contains all of the edges in $\CP$ then the restriction $\CT|_{G_{n-1}}$ is a spanning forest on the graph $G_{n-1}$ rooted at the two vertices $x$ and $y$.  Similarly, given a spanning forest $\CF_{n-1}$ on $G_{n-1}$, we can see that $\CF_{n-1} \cup \CP$ will be a spanning tree on $G_n$.  On the other hand, let $\CT_{n-1}$ be a spanning tree on $G_{n-1}$.  Then $\CT_{n-1}$ can be extended to be a spanning tree on $G_n$ in $k_n-1$ ways, as we must leave off one of the $k_n-1$ edges on the path between $x$ and $y$ while including all of the other new edges.  This proves the first identity.

Similar reasoning allows us to arrive at the second recurrence.  Fixing a vertex $x_i \in \CP$, any forest on $G_{n-1}$ rooted at $x$ and $y$ can be extended to a forest on $G_n$ rooted at $x_i$ and $x_{i+1}$ by adding the path between $x$ and $x_i$ and the path between $y$ and $x_{i+1}$.  Meanwhile, a tree on $G_{n-1}$ can be extended to a spanning forest on $G_n$ rooted at $x_i$ and $x_{i+1}$ by adding all elements of $\CP$ except for the edge connecting $x_i$ and $x_{i+1}$ and one other edge.  In particular, there will be $k_n-2$ ways to extend it.  This implies the theorem.
\end{proof}

\begin{theorem}\label{T:recur}
With notation as above, we have:
\[T(k_1,\ldots ,k_n) = k_nT(k_1,\ldots ,k_{n-1}) - T(k_1,\ldots ,k_{n-2}).\]
\end{theorem}

\begin{proof}
If we subtract the second recurrence relation in Lemma \ref{L:two} from the first, we see that $T(k_1,\ldots, k_n)-F(k_1,\ldots, k_n)=T(k_1,\ldots, k_{n-1})$ and in particular that $F(k_1,\ldots, k_n)=T(k_1,\ldots, k_n)-T(k_1,\ldots, k_{n-1})$ for all $n$.  This implies that $F(k_1,\ldots k_{n-1})=T(k_1,\ldots , k_{n-1})-T(k_1,\ldots, k_{n-2})$.  Plugging this into the first recurrence relation gives the desired result.
\end{proof}

\begin{remark}
The readers may find it strange, as the authors did at first, that the sequence $T_n$ is given by a recurrence relation that depends on the previous two terms and the length of the most recently added chain, but not on the length of the chain before it.  We note that this follows more directly from an alternative purely combinatorial proof of Theorem \ref{T:recur} that we will now sketch.

In particular, for any spanning tree $\CT_{n-1}$ on $G_{n-1}$ one can see that there are $k_{n-1}$ ways to extend it to a spanning tree on $G_n$ merely by including all but one of the new edges.  Moreover, there is a unique way to restrict it to a tree on $G_{n-2}$, by removing any portion of the tree that is contained in the $(n-1)^{st}$ polygon and adding an additional edge if the resulting set is not connected.  One can show that this gives a $k_n$-to-$1$ map from the set of trees on $G_{n-1}$ to the union of the sets of trees on $G_n$ and $G_{n-2}$, implying the theorem.
\end{remark}

\begin{example}
Let us consider the case where we have a stack of $k$-gons with $k \ge 3$, and let $T_n$ be the number of spanning trees of such a graph so that the critical group of this graph is isomorphic to $\ZZ/T_n\ZZ$.  In particular, this will be the case where $k_n$ is the constant value $k$ for all $n$, so Theorem \ref{T:recur} implies that the sequence $\{T_n\}$ satisfies the second order linear recurrence $T_n=kT_{n-1}-T_{n-2}$.  One can easily compute the initial conditions $T_0=1$ and $T_1=k$.  If one prefers an explicit formula to a recursive one, it is then possible to use well-known results on recurrence relations (see, for example, \cite[Ch. 6]{RobTes}) to compute that \[T_n=\frac{1}{2}\left[\left(1+ \frac{k}{\sqrt{k^2-4}}\right)\left(\frac{k+\sqrt{k^2-4}}{2}\right)^n + \left(1- \frac{k}{\sqrt{k^2-4}}\right)\left(\frac{k-\sqrt{k^2-4}}{2}\right)^n\right]\]

It is worth noting that when $k=4$, the graph $G_n$ is the $2$-by-$n$ grid and the number of spanning trees is computed in \cite{DM} using similar techniques to ours.
\end{example}

\begin{example}
Next, consider the example of an $n$-story `house', corresponding to the $(n+1)$-tuple $(3,4,\ldots,4)$.  As in the previous example, the number of trees will satisfy the recurrence relation $T_n = 4T_{n-1} - T_{n-2}$. One can compute by hand in this case that $T_0=3$ and $T_1 = 11$.  In particular, this shows that \[T_n=\frac{1}{2\sqrt{3}}\left[\left(3\sqrt{3}+5\right)\left(2+\sqrt{3}\right)^n+\left(3\sqrt{3}-5\right)\left(2-\sqrt{3}\right)^n\right]\]
\end{example}

\begin{example}
For our final example, we consider the case of a stack of alternating $k_1$-gons and $k_2$-gons, where $a$ and $b$ are both at least $2$.  Again, it follows from Theorem \ref{T:path} that the critical group is cyclic and therefore we only need to count the number of spanning trees to determine the group.  Let us assume that $A_n$ is the number of spanning trees of the graph formed by adding $n$ of each type of shape in an alternating fashion.  (We leave as an exercise to the reader the interesting fact that you get a different answer if you put a stack of $n$ $k_1$-gons on top of a stack of $n$ $k_2$-gons).  Moreover, let $B_n$ be the number of spanning trees of a graph composed with $n$ $k_1$-gons and $n-1$ $k_2$-gons arranged alternatingly.

In particular, it follows from Theorem \ref{T:recur} that we have $A_n=k_2B_n-A_{n-1}$ and $B_n=k_1A_{n-1}-B_{n-1}$.  From these two relations, one can deduce that $A_n=(k_1k_2-2)A_{n-1}-A_{n-2}$ and $B_n=(k_1k_2-2)B_{n-1}-B_{n-2}$.  Combined with the additional observations that $A_0=1$, $A_1=k_1k_2-1$, $B_0=0$, and $B_1=k_1$ one can use standard results on recurrence relations to get an explicit formula for the $A_n$ and $B_n$.
\end{example}

\providecommand{\bysame}{\leavevmode\hbox to3em{\hrulefill}\thinspace}
\providecommand{\MR}{\relax\ifhmode\unskip\space\fi MR }
\providecommand{\MRhref}[2]{%
  \href{http://www.ams.org/mathscinet-getitem?mr=#1}{#2}
}
\providecommand{\href}[2]{#2}

\end{document}